\documentclass[11pt]{amsart}

\usepackage[latin1]{inputenc}
\usepackage{subfigure}
\usepackage{amsfonts, amsmath, thmtools}
\usepackage{fullpage}
\usepackage[pagebackref=false,bookmarksopen=false,colorlinks,citecolor=black,linkcolor=violet]{hyperref} 
\usepackage{multirow,xcolor,pifont}
\usepackage{tikz}
\usetikzlibrary{decorations.pathreplacing, arrows,positioning}

\newtheorem{theorem}{Theorem}[section]
\newtheorem{lemma}[theorem]{Lemma}

\newtheorem{cor}[theorem]{Corollary}
\newtheorem{prop}[theorem]{Proposition}

\newtheorem*{xtheorem}{Theorem}
\theoremstyle{remark}
\declaretheorem[name=Example,qed={\lower-0.3ex\hbox{$\triangleleft$}}]{example}

\usepackage{setspace}

\newcommand{\mS}{\ensuremath{\mathcal{S}}}
\newcommand{\mV}{\ensuremath{\mathcal{V}}}
\newcommand{\bz}{\ensuremath{\mathbf{z}}}
\newcommand{\bw}{\ensuremath{\mathbf{w}}}

\newcommand{\by}{\ensuremath{\mathbf{y}}}
\newcommand{\bo}{\ensuremath{\mathbf{1}}}
\newcommand{\bi}{\ensuremath{\mathbf{i}}}
\newcommand{\ox}{\ensuremath{\overline{x}}}
\newcommand{\oz}{\ensuremath{\overline{z}}}
\newcommand{\oy}{\ensuremath{\overline{y}}}
\newcommand{\sgn}{\operatorname{sgn}}

\newcommand{\bp}{\mbox{\boldmath$\rho$}}

\newcommand{\bzer}{\ensuremath{\mathbf{0}}}
\newcommand{\bs}{\mbox{\boldmath$\sigma$}}
\newcommand{\bt}{\mbox{\boldmath$\theta$}}

\newcommand{\FunList}[3]{
    \begin{tabular}{cl} \multirow{2}{*}{\LARGE #1} & #2 \\ & #3\end{tabular}
}

\newcommand{\diag}[1]{
  \begin{tikzpicture}[scale=0.2]\makediagb{#1}\end{tikzpicture}
}

\makeatletter
\def\testbb#1{\testbb@i#1,,\@nil}%
\def\testbb@i#1,#2,#3\@nil{%
  \draw (O) --++(#1);
  \ifx\relax#2\relax\else\testbb@i#2,#3\@nil\fi}
\makeatother   

\newcommand{\makediagb}[1]{
    \coordinate (O) at (0,0); \coordinate (N) at (0,1);
    \coordinate (NE) at (1,1); \coordinate (E) at (1,0);
    \coordinate (SE) at (1,-1); \coordinate (S) at (0,-1);
    \coordinate (SW) at (-1,-1);\coordinate (W) at (-1,0);
    \coordinate (NW) at (-1,1); \coordinate (B1) at (1.2,1.2);
    \coordinate (B2) at (-1.2,-1.2);
    
    \draw (B1) --++(0,-2.4); \draw (B1) --++ (-2.4,0);
    \draw (B2) --++(0,2.4);  \draw (B2) --++ (2.4,0);   
    \testbb{#1}
}

\author{Stephen Melczer}
\author{Marni Mishna}

\title[Lattice path enumeration using diagonals]{Asymptotic lattice path enumeration using diagonals}

\address[S. Melczer]{Cheriton School of Computer Science, University of Waterloo, Waterloo ON Canada  \&  U. Lyon, CNRS, ENS de Lyon, Inria, UCBL, Laboratoire LIP}
\email{smelczer@uwaterloo.ca}
\address[M. Mishna]{Department of Mathematics, Simon Fraser University, Burnaby BC, Canada, V5A 1S6}
\email{mmishna@sfu.ca}

\begin{document}
\begin{abstract}
  This work presents new asymptotic formulas for family of walks in
  Weyl chambers. The models studied here are defined by step sets
  which exhibit many symmetries and are restricted to the first
  orthant.  The resulting formulas are very straightforward: the
  exponential growth of each model is given by the number of steps,
  while the sub-exponential growth depends only on the dimension of
  the underlying lattice and the number of steps moving forward in
  each coordinate. These expressions are derived by analyzing the
  singular variety of a multivariate rational function whose diagonal
  counts the lattice paths in question. Additionally, we show how to
  compute subdominant growth for these models, and how to determine
  first order asymptotics for excursions.
\end{abstract}
\maketitle

\section{Introduction}
\label{sec:intro}
The reflection principle and its various incarnations have been
indispensable in the study of the lattice path models, particularly in
the discovery of explicit enumerative formulas. Two examples include
the formulas for the family of reflectable walks in Weyl chambers of
Gessel and Zeilberger~\cite{GeZe92}, and various approaches using the
widely applied kernel method~\cite{Bous05, BoPe03, JaPrRe08, BoMi10}. 
In these guises,
the reflection principle is often a key element in the solution when
the resulting generating function is shown to be D-finite\footnote{A
  function is D-finite if it satisfies a linear differential equation
  with polynomial coefficients}. This is no coincidence: the
connection is an expression for the generating function as a diagonal
of a rational function. More precisely, in works such as~\cite{GeZe92,
  BoMi10}, the analysis results in generating functions expressed
as rational sub-series extractions, which can be easily converted to
diagonal expressions.  Unfortunately, the resulting explicit
representations of generating functions can be cumbersome to
manipulate. For example, much recent work on walks in Weyl chambers
has led to expressions which are determinants of large matrices with
Bessel function entries~\cite{GrMa93, Grab95, Xin10}. Here, we aim to
determine asymptotics for a family of lattice path models arising
naturally among those restricted to positive orthants -- which
correspond to walks in certain Weyl chambers -- while avoiding such
unwieldly representations.  This is acheived by working directly with
the diagonal expressions obtained through the recently developed
machinery on analytic combinatorics in several
variables~\cite{PeWi13}.

Coupling these techniques -- diagonal representation and analytic combinatorics in 
several variables -- yields explicit, yet simple, asymptotic
formulas for families of lattice paths. The focus of this article is
$d$-dimensional models whose set of allowable steps is symmetric with
respect to any axis; we call these models~\emph{highly symmetric}
walks. The techniques of analytic combinatorics
in several variables apply in a rather straightforward way to derive
dominant asymptotics for the number of walks ending anywhere and give
an effective procedure to calculate descending terms in the asymptotic
expansions. Furthermore, we also consider the subfamily of walks that return to
the origin (known as~\emph{excursions}). Once our equations are
established, they are suitable input to existing implementations such
as that of Raichev~\cite{Raic12} (however, in practice one can calculate only the
first few terms in these expansions).

The highly symmetric walks we present are amenable to a kernel method
treatment. In particular, they fit well into the ongoing study of
lattice path classes restricted to an orthant and taking only
``small'' steps~\cite{BoMi10,BoRaSa14}. This collection of models
forms a little universe exhibiting many interesting phenomena, and
recent work in two and three dimensions has used novel applications of
algebra and analysis, along with new computational techniques, to
determine exact and asymptotic enumeration formulas.  One key
predictor of the nature of a model's generating function (whether it
is rational, algebraic, or transcendental D-finite, or none of these)
is the order of a group that is associated to each model. This group
has its origins in the probabilistic study of random walks,
namely~\cite{FaIaMa99}, and when the group is finite it can sometimes be
used to write generating functions as the positive part of an explicit
multivariate rational Laurent series. The intimate relation between
the generating function of the walks and the nature of the generating
function is explored in~\cite{BoRaSa14, MeMi14a}.

For highly symmetric models in two and three dimensions, this group
coincides with that of a Weyl group for walks in the Weyl chambers
$A_1^2$ and $A_1^3$, respectively.  Indeed, one can use either
viewpoint to generalize the study of highly
symmetric models to models in arbitrary dimension.  As these viewpoints are
largely isomorphic, and the kernel method viewpoint is more
self-contained, we begin this article by working through a
straightforward generalization of the kernel method in order to write
the generating function for higher dimensional highly symmetric walks
as diagonals of rational functions.  We then perform an asymptotic
analysis of the coefficients of counting generating functions using
techniques from the study of analytic combinatorics in several
variables, and consequently link some of the combinatorial symmetries
in a walk model to both analytic properties of the generating function
and geometric properties of an associated variety.  After this is
complete, we examine how this connects to the notion of walks in Weyl
chambers, use results from their study to determine asymptotic results
about excursions, and discuss how the Weyl chamber viewpoint can be
used in future work to examine larger classes of lattice path models
through diagonals. Next we specify the walks we study in order to
precisely state our main results. 

\subsection{Highly Symmetric Walks} 
Concretely, the lattice path models we consider are restricted as
follows.  For a fixed dimension $d$, we define a model by its step set
$\mS \subseteq \{\pm1,0\}^d \setminus \{\mathbf{0}\}$ and say
that~$\mS$ is \emph{symmetric about the $x_k$ axis} if
$(i_1,\dots,i_k,\dots,i_d) \in \mS$ implies
$(i_1,\dots,-i_k,\dots,i_d) \in \mS$.  We further impose a
non-triviality condition: for each coordinate there is at least one
step in~$\mS$ which moves in the positive direction of that coordinate (this
implies that for each coordinate there is a walk in the model which moves 
in that coordinate). 

The number of walks taking steps in~$\mS$ which are restricted to the
positive orthant $\mathbb{N}^d = \mathbb{Z}_{\geq0}^d$ are studied by
expressing the counting generating functions of such models as
positive parts of multivariate rational Laurent series, which are then
converted to diagonals of rational functions in $d+1$ variables. A
first consequence is that all of these models have D-finite generating
functions (since D-finite functions are closed under the diagonal
operation).

After the above manipulations, these models are very well suited to the asymptotic
enumeration methods for diagonals of rational functions outlined
in~\cite{PeWi13}, in particular the cases which were developed by
Pemantle, Raichev and Wilson in~\cite{PeWi02} and~\cite{RaWi08}. 
Following these methods, we study the singular variety of the denominator of this 
rational function to determine related
asymptotics. The condition of having a symmetry across each axis
ensures that the variety is smooth and allows us to calculate the
leading asymptotic term explicitly.  This is not generally the case, in
our experience, and hence we focus on this particular kind of
restriction.

\begin{table}
\centering
\begin{tabular}{ | c | c @{ \hspace{0.01in} }@{\vrule width 1.2pt }@{ \hspace{0.01in} } c | c | }
  \hline
   $\mS$ & Asymptotics & $\mS$ & Asymptotics \\ \hline
  &&&\\[-5pt] 
  \diag{N,S,E,W}  & $\displaystyle \frac{4}{\pi\sqrt{1 \cdot 1}} \cdot n^{-1} \cdot 4^n = \frac4\pi \cdot \frac{4^n}n$ & 
  \diag{NE,SE,NW,SW} & $\displaystyle \frac{4}{\pi\sqrt{2 \cdot 2}} \cdot n^{-1} \cdot 4^n = \frac2\pi \cdot \frac{4^n}n$ \\[+7mm]
  \diag{N,S,NE,SE,NW,SW} & $\displaystyle \frac{6}{\pi\sqrt{3 \cdot 2}} \cdot n^{-1} \cdot 6^n = \frac{\sqrt{6}}\pi \cdot \frac{6^n}n$ &
  \diag{N,S,E,W,NW,SW,SE,NE} & $\displaystyle \frac{8}{\pi\sqrt{3 \cdot 3}} \cdot n^{-1} \cdot 8^n = \frac{8}{3\pi} \cdot \frac{8^n}n$ \\[5mm]
\hline
\end{tabular}\\[2mm]

\caption{The four highly symmetric models with unit steps in the quarter plane.} \label{tab:verif}
\end{table}

\subsection{Main results}
We present two main results in this work. The first appears as
Theorem~\ref{thm:asm}.
\begin{xtheorem}
  Let $\mS \subseteq \{-1,0,1\}^d \setminus \{\mathbf{0}\}$ be a set of unit
  steps in dimension $d$.  If $\mS$ is symmetric with respect to each
  axis, and $\mS$ takes a positive step in each direction, then the
  number of walks of length~$n$ taking steps in $\mS$, beginning at the origin, and never
  leaving the positive orthant has asymptotic expansion
\begin{equation*}
s_n  = \left[ \left(s^{(1)} \cdots s^{(d)}\right)^{-1/2} \pi^{-d/2} |\mathcal{S}|^{d/2}\right] \cdot n^{-d/2} \cdot |\mathcal{S}|^n + O\left( n^{-(d+1)/2} \cdot |\mathcal{S}|^n \right),
\end{equation*}
where $s^{(k)}$ denotes the number of steps in $\mS$ which have $k^\text{th}$ coordinate 1.
\end{xtheorem}

This formula is easy to apply to any given model, and for
certain infinite families as well.

\begin{example} 
  When $d=2$ there are four non-isomorphic highly symmetric walks in
  the quarter plane, listed in Table~\ref{tab:verif}.  Applying
  Theorem~\ref{thm:asm} verifies the asymptotic results guessed
  previously by~\cite{ BoKa09}.
\end{example}

\begin{example}Let $\mS = \{-1,0,1\}^d \setminus \{\mathbf{0}\}$, the full
set of possible steps. This is symmetric across each axis. We compute
that $|\mathcal{S}|=3^d-1$, and $s^{(j)}=3^{d-1}$ for all $j$ and so
\[
s_n \sim \left( \frac{(3^d-1)^{d/2}}{3^{d(d-1)/2}\cdot\pi^{d/2}}\right) \cdot n^{-d/2} \cdot (3^d-1)^n. 
\]
\end{example}

\begin{example} Let~$e_k = (0,\dots,0,1,0,\dots,0)$ be the
  $k^\text{th}$ standard basis vector in~$\mathbb{R}^d$, and consider the set of steps
  \mbox{$\mS=\{e_1,-e_1,\dots,e_d,-e_d\}$}. Then the number of walks of
  length $n$ taking steps from~$\mS$ and never leaving the positive
  orthant has asymptotic expansion
\begin{equation*} 
s_n \sim \left(\frac{2d}{\pi}\right)^{d/2}\,n^{-d/2} \, (2d)^n.
\end{equation*}
\end{example}

The second main result is a comparable statement for excursions,
Theorem~\ref{thm:excursion}.
\begin{xtheorem}
Let $\mS \subseteq \{-1,0,1\}^d \setminus \{\mathbf{0}\}$ be a set of unit steps in dimension $d$.  If $\mS$ is symmetric with respect to each axis, and $\mS$ takes a positive step in each direction, then the number of walks $e_n$ of length $n$ taking steps in $\mS$, beginning and ending at the origin, and never leaving the positive orthant satisfies
\[ e_n = O\left(\frac{|\mS|^n}{n^{3d/2}}\right). \]
\end{xtheorem}

\subsection{Organization of the paper}
The article is organized as follows.  Section~\ref{sec:lwalks}
describes how to express the generating function using an orbit sum by
applying the kernel method, following the strategy described
in~\cite{BoMi10}.  We then derive Equation~\eqref{eq:rat}, which
describes the generating function as the diagonal of a rational power
series in multiple variables.  Section~\ref{sec:acsv} justifies why
the work of Pemantle and Wilson~\cite{PeWi13} is applicable, with the
asymptotic results computed in Section~\ref{sec:results}.  We discuss
the sub-dominant growth, and compute an example in
Section~\ref{sec:lot}. Section~\ref{sec:tele} discusses the
differential equations satisfied by these generating functions, and
how to use creative telescoping techniques to find them.  We tabulate
some small examples. We conclude with a discussion of how these walks
fit into the context of walks in Weyl chambers, which allows us to
obtain results on the asymptotics of walk excursions, and also to
consider other families of walks.

\begin{figure}
\centering
\tikzstyle{vspecies}=[align=center,rectangle, minimum size=0.5cm,text width = 19em,draw=black,fill=white, rounded corners,thick]
\begin{tikzpicture}[auto, outer sep=1pt, node distance=1.5cm]
\node [vspecies] (B) {Determine functional equation for $F(\bz,t)$} ;
\node [vspecies, below = 0.32cm of B] (C) {\FunList{}{Represent $F(\bz,t)$ as the}{positive part of rational $R(\bz,t)$}} ;
\node [vspecies, below of = C] (D) {\FunList{}{Convert to a diagonal}{extraction of $G(\by,t)/H(\by,t)$}} ;
\node [vspecies, below = 0.32cm of D] (E) {Find the critical points of $\mathbb{V}(H)$};
\node [vspecies, below = 0.32cm of E] (F) {Refine to minimal points of $\mathbb{V}(H)$};
\node [vspecies, below = 0.32cm of F] (G) {\FunList{}{Find asymptotics
    via formulas}{of Pemantle and Wilson~\cite{PeWi13}}} ;
\draw [<-,thick] (C) --  node {} (B) ;
\draw [<-,thick] (D) --  node {} (C) ;
\draw [<-,thick] (E) --  node {} (D) ;
\draw [<-,thick] (F) --  node {} (E) ;
\draw [->,thick] (F) --  node {} (G) ;
\end{tikzpicture}
\caption{The strategem of determining asymptotics via the generalized
  kernel method for symmetric walks. }
\end{figure}
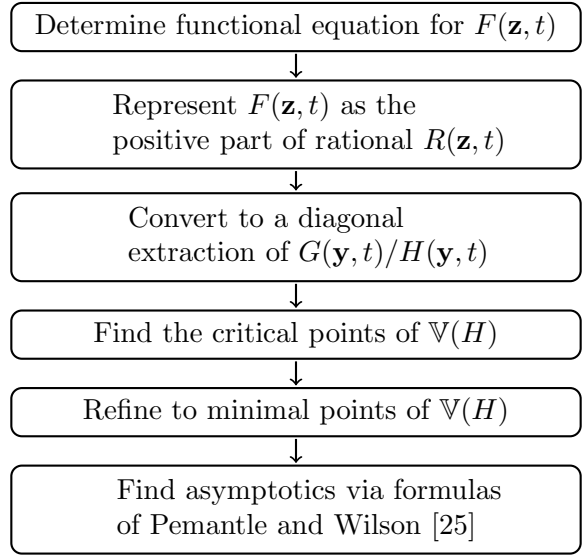

\section{Deriving a diagonal expression for the generating function}
\label{sec:lwalks}
Fix a dimension $d$ and a highly symmetric set of steps~$\mS \subseteq
\{\pm1,0\}^d \setminus \{\mathbf{0}\}$. Recall this means that
$(i_1,\dots,i_k,\dots,i_d) \in \mS$ implies
$(i_1,\dots,-i_k,\dots,i_d) \in \mS$. In this section we derive a
functional equation for a multivariate generating function, apply the
orbit sum method to derive a closed expression related to this
generating function, and conclude by writing the univariate counting
generating function for the number of walks as the complete diagonal
of a rational function.

The following notation is used throughout: 
\[ \oz_i= z_i^{-1}; \qquad \bz = (z_1, \dots, z_d); \qquad \bi =( i_1,i_2,\dots,i_d)\in \mathbb{Z}^d; \qquad \bz^{\bi} = z_1^{i_1}\cdots z_d^{i_d}, \]
and we write $\mathbb{Q}[z_k,\overline{z_k}]$ to refer to the ring of Laurent polynomials in the variable $z_k$.

\subsection{A functional equation}
To begin, we define the generating function:
\begin{equation}
  F(\bz,t)= \sum_{\substack{n \geq 0 \\ \bi \in \mathbb{Z}^d}} s_{\bi}(n) \bz^{\bi} t^n 
  = \sum_{n \geq 0} \left(\sum_{\bi \in \mathbb{Z}^d}s_{\bi}(n) z_1^{i_1}\cdots z_d^{i_d}\right) t^n \in \mathbb{Q}[z_1,\oz_1,\dots,z_d,\oz_d][\![t]\!], 
\end{equation}
where $s_{\bi}(n)$ counts the number of walks of length $n$ taking
steps from $\mS$ which stay in the positive orthant and end at lattice
point $\bi \in \mathbb{Z}^d$.  Note that the series $F(\bo,t)$ is the generating
function for the total number of walks in the orthant, and we can recover the series 
for walks ending on the hyperplane $z_k=0$ by setting $z_k=0$ in the
series~$F(\bz,t)$ (the variables
$z_1,\dots,z_d$ are referred to as \emph{catalytic} variables in the
literature, as they are present during the analysis and removed at the
end of the `reaction' via specialization to 1).  We also define the
function (known as either the \emph{characteristic polynomial} or the
\emph{inventory} of $\mS$) by
\begin{equation} 
S(\bz)= \sum_{\bi \in \mS} \bz^{\bi} = [t^1] F(\bz,t) \in \mathbb{Q}\left[z_1,\oz_1,\dots,z_d,\oz_d\right]. 
\end{equation}
In many recent analyses of lattice walks, functional equations are
derived by translating the following description of a walk into a
generating function equation: a walk is either an empty walk, or a
shorter walk followed by a single step. To ensure the condition that
the walks remain in the positive orthant, we must not count walks that
add a step with a negative $k$-th component to a walk ending on
the hyperplane $z_k=0$. To account for this, it is sufficient to subtract
an appropriate multiple of $F$ from the functional equation: 
$t\oz_k F(z_1, \dots, z_{k-1}, 0, z_{k+1},\dots, z_d, t)$, 
however if a given step has several negative
components we must use inclusion and
exclusion to prevent over compensation.
 
This can be made explicit. Let $\mS \subseteq
  \{1,0,-1\}^d$ define a $d$-dimensional lattice model restricted to the
  first orthant, and let~$F(\bz,t)$ be the generating function for
  this model, counting the number of
  walks of length $n$ with marked endpoint. Let
  $V=\{1, \dots, d\}$, so that it is the set of coordinates $j$ for which
  there is at least one step in $\mS$ with $-1$ in the $j$-th
  coordinate (this is the full set of indices by our assumptions). 
  Then, by translating the combinatorial recurrence described above,
 we see that $F(\bz,t)$ satisfies the functional equation
\begin{align}
\begin{split}
(z_1\cdots z_d)F(\bz,t) = (z_1\cdots z_d) &+ t(z_1\cdots z_d)S(\bz) F(\bz,t)\\
& - t\sum_{V'\subseteq V} (-1)^{|V'|} \left[(z_1\cdots z_d)S(\bz) F(\bz,t)\right]_{\{z_j=0: j\in V'\}}.
\end{split}
\end{align}
Basic manipulations then give the following result.

\begin{lemma}
Let $F(\bz,t)$ be the multivariate generating function described above.  Then 
\begin{equation} 
(z_1\cdots z_d)\left(1-tS(\bz)\right)F(\bz,t) = (z_1\cdots z_d) + \sum_{k=1}^d A_k(z_1,\dots,z_{k-1},z_{k+1}\dots,z_d,t), \label{eq:kernel}
\end{equation}
for some $A_k \in \mathbb{Q}[z_1,\dots,z_{k-1},z_{k+1},\dots,z_d][\![t]\!]$.
\end{lemma}

\begin{example} Set $\mS=\{e_1, -e_1, \dots, e_d, -e_d\}$. In this
  case $S(\bz)=\sum_{j=1}^d (z_j+\oz_j)$, so $(z_1\cdots z_d)S(\bz)$ vanishes
  when at least two of the $z_j$ are zero, and the generating function satisfies
 \[(z_1\cdots z_d)\left(1-tS(\bz)\right)F(\bz,t) = (z_1\cdots z_d) + \sum_{k=1}^d t(z_1\dots z_{k-1} z_{k+1} \dots z_d) F(z_1, \dots, z_{j-1}, 0, z_{j+1}\dots, z_d).\]
\end{example}

\subsection{The Orbit Sum Method}
\label{sec:orbit}
The orbit sum method, when it applies, has three main steps: find a
suitable group~$\mathcal{G}$ of rational maps; apply the elements of
the group to the functional equation and form a telescoping sum; and
(ultimately) represent the generating function of a model as the
positive series extraction of an explicit rational function.
Bousquet-M{\'e}lou and Mishna~\mbox{\cite{BoMi10}} illustrate the applicability in the case of
lattice walks, and it has been adapted to several dimensions~\cite{BoBoKaMe14}.

\subsubsection{The group $\mathcal{G}$}
For any $d$-dimensional model,  we define the group $\mathcal{G}$ of $2^d$ rational maps by
\begin{equation}
  \mathcal{G} := \left\{(z_1,\dots,z_d) \mapsto (z_1^{i_1},\dots,z_d^{i_d}) : (i_1,\dots,i_d) \in \{-1,1\}^d \right\}.
\end{equation} 
Given $\sigma \in \mathcal{G}$, we can consider $\sigma$ as a map on
$\mathbb{Q}[z_1,\oz_1,\dots,z_d,\oz_d][\![t]\!]$ through the group action
defined by \mbox{$\sigma\left(A(\bz,t)\right) :=
  A\left(\sigma(\bz),t\right)$}.  Due to the symmetry of the step set
across each axis, one can verify that
$\sigma(S(\bz))=S(\sigma(\bz))=S(\bz)$ always holds.  The fact that
this group does not depend on the step set of the model -- only on the
dimension $d$ -- is crucial to obtaining the general results here.
When $d$ equals two, the group $\mathcal{G}$ matches the group used
by~\cite{FaIaMa99} and~\cite{BoMi10}. As we will see in Section~\ref{sec:Weyl}, 
$\mathcal{G}$ corresponds to the Weyl group of the Weyl chamber $A_1^d$, where the step set
$\mS$ can be studied in the context of Gessel and Zeilberger~\cite{GeZe92}.

\subsubsection{A telescoping sum}
Next we apply each of the $2^d$ elements of $\mathcal{G}$ to
Equation~\eqref{eq:kernel}, and take a weighted
sum.  Define~$\sgn(\sigma)=(-1)^r$, where $r = \# \{k : \sigma(z_k)=\oz_k\}$, and let $\sigma_k$ 
be the map which sends $z_k$ to $\oz_k$ and fixes all other components of $(z_1,\dots,z_d)$.

\begin{lemma} 
  Let $F(\bz,t)$ be the generating function counting the number of
  walks of length $n$ with marked endpoint.  Then, as elements of the
  ring $\mathbb{Q}[z_1,\oz_1,\dots,z_d,\oz_d][\![t]\!]$,
\begin{equation} 
  \sum_{\sigma \in \mathcal{G}}\sgn(\sigma)\cdot \sigma(z_1\cdots z_d) \sigma(F(\bz,t)) = \frac{\sum_{\sigma \in \mathcal{G}}\sgn(\sigma)\cdot \sigma(z_1\cdots z_d)}{1-tS(\bz) }. \label{eq:orbit}
\end{equation}
\end{lemma}

\begin{proof}
For each $\sigma \in \mathcal{G}$ we have $\sgn(\sigma) =
-\sgn(\sigma_k\sigma)$ and, for the $A_k$ in Equation~\eqref{eq:kernel},
\begin{equation*} 
  \sigma(A_k(z_1,\dots,z_{k-1},z_{k+1}\dots,z_d,t)) =
  (\sigma_k\sigma)(A_k(z_1,\dots,z_{k-1},z_{k+1}\dots,z_d,t)). 
\end{equation*}
Thus, we can apply each $\sigma \in \mathcal{G}$ to Equation
\eqref{eq:kernel} and sum the results, weighted by $\sgn(\sigma)$, to cancel
each $A_k$ term on the right hand side.  Minor algebraic manipulations, along with the
fact that the group elements fix $S(z_1,\dots,z_d)$, then give
Equation \eqref{eq:orbit}.
\end{proof}

\subsubsection{Positive series extraction}
Next, we note that each term in the expansion of
\[ \sigma_1(z_1,\dots,z_d) \sigma_1(F(\bz,t)) = - (\oz_1z_2\cdots z_d) F(\oz_1,z_2,\dots,z_d,t) \in \mathbb{Q}[z_1,\oz_1,\dots,z_d,\oz_d][\![t]\!]\]
 has a negative power of $z_1$. In fact, except for when $\sigma$ is the identity any summand $\sigma(z_1\cdots z_d) \sigma(F(\bz,t))$ on the left hand side of Equation \eqref{eq:orbit} contains a negative power of at least one variable in any term of its expansion.  

 With this in mind, for an element $A(\bz,t) \in \mathbb{Q}[z_1,\oz_1,\dots,z_d,\oz_d][\![t]\!]$ we let  
$[z_k^{\geq}]A(\bz,t)$ denote the sum of all terms of $A(\bz,t)$ which contain only non-negative powers of $z_k$. 
Lemma~\ref{thm:pospart} then follows from the identity
\[ \sum_{\sigma \in \mathcal{G}}\sgn(\sigma)\cdot \sigma(z_1\cdots z_d) = (z_1-\oz_1)\cdots(z_d-\oz_d), \]
 which can be proven by induction. 

\begin{lemma} \label{thm:pospart}
Let $F(\bz,t)$ be the generating function counting the number of walks of length~$n$ with marked endpoint.  Then
\begin{equation}
F(\bz,t) = [z_1^{\geq}]\cdots[z_d^{\geq}]R(\bz,t), \label{eq:pospart}
\end{equation}
where
\begin{equation*}
R(\bz,t) = \frac{(z_1-\oz_1)\cdots(z_d-\oz_d)}{(z_1\cdots z_d)(1-tS(\bz)) }.
\end{equation*}
\end{lemma}

Since the class of D-finite functions is closed under positive series extraction -- as shown in \cite{Li89} -- an immediate consequence is the following. 
\begin{cor}\label{cor:Dfinite}
Under the above conditions on \mS, the generating functions $F(\bz, t)$ and (thus)
$F(\bo, t)$ are D-finite functions. 
\end{cor}

\subsection{The generating function as a diagonal}

Given an element 
\begin{equation*} 
B(\bz,t) = \sum_{n \geq 0}\left(\sum_{\bi \in \mathbb{Z}^d} b_{\bi}(n)z_1^{i_1}\cdots z_d^{i_d}\right)t^n  \in \mathbb{Q}[z_1,\oz_1,\dots,z_d,\oz_d][\![t]\!],
\end{equation*}
we let $\Delta$ denote the (complete) diagonal operator
\begin{equation*} 
\Delta B(\bz,t) := \sum_{n \geq 0} b_{n,\dots,n}(n) t^n.
\end{equation*}

There is a natural correspondence between the diagonal operator and extracting the positive part of a multivariate power series, as in Equation~\eqref{eq:pospart}.

\begin{prop} \label{prop:diag}
Let $B(\bz,t)$ be an element of $\mathbb{Q}[z_1,\oz_1,\dots,z_d,\oz_d][\![t]\!]$.  Then 
\begin{equation} \label{eq:postodiag}
[z_1^{\geq}]\cdots[z_d^{\geq}]B(\bz,t) \bigg|_{z_1=1,\dots,z_d=1} = \Delta \left(\frac{B\left(\oz_1,\dots,\oz_d,z_1\cdots z_d\cdot t\right)}{(1-z_1)\cdots(1-z_d)}\right).
\end{equation}
\end{prop}

\begin{proof}
Suppose that $B$ has the expansion
\begin{equation*}
B(\bz,t) = \sum_{n \geq 0} \left( \sum_{\bi \in \mathbb{Z}^d} b_{\bi}(n)z_1^{i_1}\cdots z_d^{i_d} \right)t^n.
\end{equation*}
Then the right hand side of Equation~\eqref{eq:postodiag} is given by 
\begin{equation*}
\Delta \left(\sum_{k \geq 0} z_1^k \right)\cdots\left(\sum_{k \geq 0} z_d^k \right)
\left(\sum_{n \geq 0} \left( \sum_{\bi \in \mathbb{Z}^d} b_{\bi}(n) z_1^{n-i_1}\cdots z_d^{n-i_d} \right)t^n\right)
\end{equation*}
so that the coefficient of $t^n$ in the diagonal is the sum of all terms $b_{\bi}(n)$ with $i_1,\dots,i_d\geq 0$ (by assumption there are only finitely many which are non-zero).  But this is exactly the coefficient of $t^n$ on the left hand side.
\end{proof}

We note also that in the context of lattice path models with step set $\mS \subseteq \{\pm1,0\}^d \setminus \{\mathbf{0}\}$, the modified generating function $F\left(\oz_1,\dots,\oz_d, z_1\cdots z_d\cdot t\right)$ is actually a power series in the variables $z_1,\dots,z_d,t$ (as a walk cannot move farther on the integer lattice than its number of steps).  Combining Lemma~\ref{thm:pospart} and Proposition~\ref{prop:diag} implies that the generating function for the number of walks  can be represented as $F(\bo,t) = \Delta
\left(\frac{G(\bz,t)}{H(\bz,t)}\right)$, where
\begin{align}
\frac{G(\bz,t)}{H(\bz,t)} &= \frac{(1-z_1^2)\cdots(1-z_d^2)}{1 - t(z_1\cdots z_d)S(\bz)}\cdot \frac{1}{(1-z_1)\cdots(1-z_d)}\notag \\
&= \frac{(1+z_1)\cdots(1+z_d)}{1 - t(z_1\cdots z_d)S(\bz)}. \label{eq:rat}
\end{align}
To be precise, $G(\bz,t)$ and $H(\bz,t)$ are defined as the numerator and denominator of Equation~\eqref{eq:rat}.

\begin{example} For the walks defined by $\mS=\{e_1, -e_1, \dots, e_d, -e_d\}$, we have
\[ \frac{G(\bz,t)}{H(\bz,t)} = \frac{(1+z_1)\cdots(1+z_d)}{1 - t\sum_{k=1}^n (1+z_k^2)(z_1\cdots z_{k-1}z_{k+1}\cdots z_d)}.\] 
Note that this rational function is not unique, in the sense that there are other rational functions whose diagonals yield the same counting sequence. 
\end{example}

\subsection{The singular variety associated to the kernel}

Here, we pause to note that the \emph{combinatorial\/} symmetries of
the step sets that we consider affect the \emph{geometry} of the
variety of $H(\bz,t)$ -- called the \emph{singular variety}. This has
a direct impact on both the asymptotics of the counting sequence under
consideration and the ease with which its asymptotics are computed.
In particular, any factors of the form $(1-z_k)$ present in the
denominator of this rational function before simplification could have
given rise to non-simple poles and thus
made the singular variety non-smooth.  Although non-smooth varieties
can be handled in many cases -- see~\cite{PeWi13} -- having a smooth
singular variety is the easiest situation in which one can work in the
multivariate setting.  Understanding the interplay between
the step set symmetry and the singular variety geometry, and in the
process dealing with the non-smooth cases, is promising future work.

\section{Analytic combinatorics in several variables}
\label{sec:acsv}
Following the work of Pemantle and Wilson~\cite{PeWi02} and Raichev
and Wilson~\cite{RaWi08}, we can determine the dominant asymptotics
for the diagonal of the multivariate power series $
\frac{G(\bz,t)}{H(\bz,t)}$ by studying the variety (complex set of
zeroes) $\mV \subseteq \mathbb{C}^{d+1}$ of the denominator
\begin{equation*} H(\bz,t) = 1 - t\cdot(z_1\cdots z_d)S(\bz). \end{equation*} 
To begin, a particular set of singular points -- called the \emph{critical points} -- containing all singular points which \emph{could} affect the asymptotics of $\Delta (G/H)$ are computed in Section~\ref{sec:crtpts}. The set of critical points is then refined to those which determine the dominant asymptotics up to an exponential decay in Section~\ref{sec:minpts}; this refined set is called the set of \emph{minimal points} as they are the critical points which are `closest' to the origin in a sense made precise below.  The enumerative results come from calculating a Cauchy residue type integral, and after determining the minimal points we determine asymptotics in Section~\ref{sec:results} using pre-computed formulas for such integrals which can be found in \cite{PeWi13}.  In fact, up to polynomial decay there is only one singular point which determines dominant asymptotics for each model -- the point $\bp = (\bo,1/|\mS|)$ -- and this uniformity aids greatly in computing the quantities required in the analysis of a general step set, in order to obtain Theorem~\ref{thm:asm}.

We first verify our claim in the previous section that the variety
is smooth (that is, at every point on $\mV$ one of the partial
derivatives $H_{z_k}$ or $H_t$ does not vanish).  Indeed, any
non-smooth point on $\mV$ would have to satisfy both
\begin{align*}
1-t(z_1\cdots z_d)S(\bz) &= H = 0\\
\text{and} \quad -(z_1\cdots z_d)S(\bz) &= H_t = 0,
\end{align*}
which can never occur.  Equivalently, this shows that at each point in $\mV$ there exists a neighbourhood $N \subseteq \mathbb{C}^{d+1}$ such that $\mV \cap N$ is a complex submanifold of $N$.

\subsection{Critical points} 
\label{sec:crtpts}
The next step is to find the critical points.  Determined through
an appeal to stratified Morse theory, for a smooth variety the
critical points are precisely those which satisfy the following \emph{critical
  point equations}:
\begin{equation*}
H = 0, \quad tH_t = z_1H_{z_1},\quad tH_t = z_2H_{z_2},\quad \dots\quad   tH_t = z_dH_{z_d},
\end{equation*}
which we now solve.  Given $\bz \in \mathbb{C}^d$, define 
\[ \bz_{\hat{k}} := (z_1,\dots,z_{k-1},z_{k+1},\dots,z_d) \in \mathbb{C}^{d-1}. \] 
As each step in~$\mS$ has coordinates taking values in~$\{-1,0,1\}$, we may
collect the coefficients of the $k^\text{th}$ variable, and
use the symmetries present to write 
\begin{equation}
 S(\bz)= (\oz_k + z_k) S_{1}^{(k)}(\bz_{\hat{k}}) + S_{0}^{(k)}(\bz_{\hat{k}}), \label{eq:inv}
\end{equation}
which uniquely defines the Laurent polynomials $S_1^{(k)}(\bz_{\hat{k}}) $ and $S_0^{(k)}(\bz_{\hat{k}}) $.  With this notation the equation \mbox{$tH_t = z_kH_{z_k}$} becomes
\begin{equation*}
 t(z_1\cdots z_d)S(\bz) = t(z_1\cdots z_d)S(\bz) + t(z_1\cdots z_d)(z_kS_{z_k}(\bz)), 
 \end{equation*}
which implies
\begin{equation}
0 = t(z_1\cdots z_d)\cdot z_kS_{z_k}(\bz)= t\left( z_k^2 -1\right)(z_1\cdots z_{k-1}z_{k+1}\cdots z_d)S_1^{(k)}(\bz_{\hat{k}}).  \label{eq:CPeqn}
\end{equation}
Note that while $(z_1\cdots z_{k-1}z_{k+1}\cdots z_d)S_1^{(k)}(\bz_{\hat{k}})$ is a polynomial, 
$S_{1}^{(k)}(\bz_{\hat{k}})$ itself is a Laurent polynomial, so one must be careful when specializing variables to 0 in the expression.  This calculation characterizes the critical points of $\mV$.

\begin{prop} \label{prop:critpt}
The point $(\bz,t) = (z_1,\dots,z_d,t) \in \mV$ is a critical point of $\mV$ if and only if for each $1\leq k\leq d$ either:
\begin{enumerate}
\item $z_k= \pm1$ or,
\item the polynomial $(y_1\cdots y_{k-1}y_{k+1}\cdots y_d)S_1^{(k)}(\by_{\hat{k}})$ has a root at $\bz$.
\end{enumerate} 
\end{prop}
\begin{proof}
We have shown above that the critical point equations reduce to Equation~\eqref{eq:CPeqn}.  Furthermore, if $t$~were zero at a point on $\mV$ then $0=H(z_1,\dots,z_n,0)=1$, a contradiction.
\end{proof}

It is interesting to note that the polynomial $(y_1\cdots y_{k-1}y_{k+1}\cdots y_d)S_1^{(k)}(\by_{\hat{k}})$ has combinatorial signifigance, as the subset of $S(\bz)$ which encodes only the steps which move forwards in their $k^\text{th}$ coordinate.

\subsection{Minimal points} 
\label{sec:minpts}
Among the critical points, only those which are `closest' to the
origin will contribute to the asymptotics, up to an exponentially
decaying error.  This is analogous to the single variable case, where
the singularities of minimum modulus are those which contribute to the
dominant asymptotic term.  To be precise, for any point $(\bz,t)
\in\mathbb{C}^{d+1}$ we define the closed polydisk
\[ D(\bz,t) := \{(\mathbf{w},t') \in \mathbb{C}^{d+1} : |t'|\leq|t|
\text{ and } |w_j|\leq|z_j| \text{ for } j=1,\dots,d\}.\] The critical
point $(\mathbf{z},t)$ is called \emph{strictly minimal} if $D(\bz,t)
\cap \mV = \{(\bz,t)\}$, and \emph{finitely minimal} if the
intersection contains only a finite number of points, all of which are
on the boundary of $D(\bz,t)$.  Finally, we call a critical point
\emph{isolated} if there exists a neighbourhood of $\mathbb{C}^{d+1}$
where it is the only critical point.  In our case, we need only be
concerned with isolated finitely minimal points.

\begin{prop} \label{prop:minpt} 
The point $\bp = (\bo,1/|\mS|)$ is a
  finitely minimal point of the variety $\mV$.  Furthermore, any point
  in $D(\bp) \cap \mV$ is an isolated critical point.
\end{prop}
\begin{proof}
The point $\bp$ is critical as it lies on $\mV$ and its first $d$ coordinates are all one.  Suppose $(\bw,t_\bw)$ lies in $\mV \cap D(\bp)$, where we note that any choice of $\bw$ uniquely determines $t_\bw$ on $\mV$.  Then, as $t_\bw \neq 0$, 
\begin{equation*}
\left| \sum_{(i_1,\dots,i_d) \in \mS}w_1^{i_1+1}\cdots w_d^{i_d+1} \right| = \bigg| (w_1\cdots w_d)S(\bw)\bigg| = \left|\frac{1}{t_\bw}\right| \geq |\mS|.
\end{equation*}
But $(\bw,t_\bw) \in D(\bp)$ implies $|w_j|\leq1$ for each $1\leq j \leq d$.  Thus, the above inequality states that the sum of $|\mS|$ complex numbers of modulus at most one has modulus $|\mS|$.  The only way this can occur is if each term in the sum has modulus one, and all terms point in the same direction in the complex plane.  By symmetry, and the assumption that we take a positive step in each direction, there are two terms of the form $w_2^{i_2+1}\cdots w_d^{i_d+1}$ and $w_1^2w_2^{i_2+1}\cdots w_d^{i_d+1}$ in the sum, so that $w_1^2$ must be 1 in order for them to point in the same direction.  This shows $w_1=\pm1$, and the same argument applies to each $w_k$, so there are at most $2^d$ points in $\mV \cap D(\bp)$.

By Proposition~\ref{prop:critpt} every such point $(\bw,t_\bw) \in \mV \cap D(\bp)$ is critical, and to show it is isolated it is sufficient to prove $S_1^{(k)}(\bw_{\hat{k}}) \neq 0$ for all $1 \leq k \leq d$.  Indeed, if $S_1^{(k)}(\bw_{\hat{k}})=0$ then $\bw \in \mV$ implies
\[ |t_\bw| =\frac{1}{\left| w_1\cdots w_d S_0^{(k)}(\bw_{\hat{k}}) \right|} \geq \frac{1}{\left|S_0^{(k)}(\bw_{\hat{k}})\right|} \geq \frac{1}{S_0^{(k)}(\bo)} > \frac{1}{|\mS|},  \]
by our assumption that $\mS$ contains a step which moves forward in the $k^\text{th}$ coordinate.  This contradicts  $(\bw,t_\bw) \in D(\bp)$.
\end{proof}

\subsection{Asymptotics Results}
\label{sec:results}
To apply the formulas of~\cite{PeWi02} we need to define a few quantities.  To start, we note that on all of $\mV$ we may parametrize the coordinate $t$ as 
\begin{equation*} t(\bz) = \frac{1}{z_1\cdots z_d S(\bz)}. \end{equation*}
For each point $(\bw,t_\bw) \in \mV \cap D(\bp)$, the analysis of \cite{PeWi02} shows that the asymptotics of the integral in question which determines asymptotics for a given model depends on the function
\begin{align}
\tilde{f}^{(\bw)}(\mbox{\boldmath$\theta$}) &= \log\left( \frac{t(w_1e^{i\theta_1},\dots,w_d e^{i\theta_d})}{t_{\bw}}\right) + i \sum_{k=1}^d \theta_k \notag \\
&= \log\left( \frac{ S(\bw)}{e^{i(\theta_1+\cdots+\theta_d)}S(w_1e^{i\theta_1},\dots,w_de^{i\theta_d})}\right) + i(\theta_1+\cdots+\theta_d) \notag \\
&= \log S(\bw) - \log S(w_1 e^{i\theta_1},\dots,w_d e^{i\theta_d}). \label{eq:ftil}
\end{align}
Let $\mathcal{H}_{\bw}$ denote the determinant of the Hessian of $\tilde{f}^{(\bw)}(\mbox{\boldmath$\theta$})$ at $\mathbf{0}$:
\[ \mathcal{H}_{\bw} := \det \tilde{f''}^{(\bw)}(\mathbf{0}) = \begin{vmatrix} 
\tilde{f}^{(\bw)}_{\theta_1 \theta_1}(\mathbf{0}) &  \tilde{f}^{(\bw)}_{\theta_1 \theta_2}(\mathbf{0}) & \cdots & \tilde{f}^{(\bw)}_{\theta_1 \theta_d}(\mathbf{0}) \\[+2mm]
\tilde{f}^{(\bw)}_{\theta_2 \theta_1}(\mathbf{0}) &  \tilde{f}^{(\bw)}_{\theta_2 \theta_2}(\mathbf{0}) & \cdots & \tilde{f}^{(\bw)}_{\theta_2 \theta_d}(\mathbf{0}) \\[+2mm]
\vdots & \vdots & \ddots & \vdots \\[+2mm]
\tilde{f}^{(\bw)}_{\theta_d \theta_1}(\mathbf{0}) &  \tilde{f}^{(\bw)}_{\theta_d \theta_2}(\mathbf{0}) & \cdots & \tilde{f}^{(\bw)}_{\theta_d \theta_d}(\mathbf{0}) 
 \end{vmatrix}, \]
 if $\mathcal{H}_{\bw} \neq 0$ then we say $(\bw,t_\bw)$ is \emph{non-degenerate}.  The main asymptotic result of smooth multivaritate analytic combinatorics, in this restricted context, is the following (the original result allows for asymptotic expansions of coefficient sequences more generally defined from multivariate functions than the diagonal sequence).

\begin{theorem}[Adapted from Theorem 3.5 of \cite{PeWi02}] \label{thm:expGen}
Suppose that the meromorphic function $F(\bz,t) = G(\bz,t)/H(\bz,t)$ has an isolated strictly minimal simple pole at $(\bw,t_\bw)$.  If $tH_t$ does not vanish at $(\bw,t_\bw)$ then there is an asymptotic expansion
\begin{equation} c_n \sim (w_1\cdots w_d \cdot t)^{-n} \sum_{l \geq l_0} C_l n^{-(d+l)/2} \label{eq:Genasm}\end{equation}
for constants $C_l$, where $l_0$ is the degree to which $G$ vanishes near $(\bw,t_\bw)$.  When $G$ does not vanish at $(\bw,t_\bw)$ then $l_0 = 0$ and the leading term of this expansion is
\begin{equation} 
C_0 = (2\pi)^{-d/2}\mathcal{H}_{\bw}^{-1/2} \cdot \frac{G(\bw,t_\bw)}{t H_t(\bw,t_\bw)}.  \label{eq:domterm}
\end{equation}
\end{theorem}

In fact, Corollary 3.7 of \cite{PeWi02} shows that in the case of a finitely minimal point one can simply sum the contributions of each point.  Combining this with the above calculations gives our main result.

\begin{theorem}\label{thm:asm}
Let $\mS \subseteq \{-1,0,1\}^d \setminus \{\mathbf{0}\}$ be a set of unit steps in dimension $d$.  If $\mS$ is symmetric with respect to each axis, and $\mS$ takes a positive step in each direction, then the number of walks of length $n$ taking steps in $\mS$, beginning at the origin, and never leaving the positive orthant has asymptotic expansion
\begin{equation} \label{eq:thmexp}
s_n  = \left[ \left(s^{(1)} \cdots s^{(d)}\right)^{-1/2} \pi^{-d/2} |\mathcal{S}|^{d/2}\right] \cdot n^{-d/2} \cdot |\mathcal{S}|^n + O\left( n^{-(d+1)/2} \cdot |\mathcal{S}|^n \right),
\end{equation}
where $s^{(k)}$ denotes the number of steps in $\mS$ which have $k^\text{th}$ coordinate 1.
\end{theorem}
\begin{proof}
We begin by verifying that each point $(\bw,t_\bw) \in \mV \cap D(\bp)$ satisfies the conditions of Theorem~\ref{thm:expGen}:

\begin{description}
\item[1. $(\bw,t_\bw)$ is a simple pole] As $\mV$ is smooth, the point $(\bw,t_\bw)$ is a simple pole.
\item[2. $(\bw,t_\bw)$ is isolated] This is proven in Proposition~\ref{prop:minpt}.
\item[3. $tH_t$ does not vanish at $(\bw,t_\bw)$] This follows from $t_\bw H_t(\bw,t_\bw) = 1/(w_1\cdots w_d) \neq 0$.
\item[4. $(\bw,t_\bw)$ is non-degenerate]  Directly taking partial derivatives in Equation~\eqref{eq:ftil} implies
\begin{equation*}
  \tilde{f}^{(\bw)}_{\theta_j\theta_k} (\mathbf{0})
     = \left\{
         \begin{array}{lr}
            \displaystyle w_jw_k\frac{S_{y_jy_k}(\bw) S(\bw) - S_{y_j}(\bw)S_{y_k}(\bw) }{S(\bw)^2} & :  \displaystyle j \neq k\\\\
            \displaystyle \frac{S_{y_jy_j}(\bw) S(\bw) + w_jS_{y_j}(\bw) S(\bw) - S_{y_j}(\bw)^2}{S(\bw)^2}  & : j = k
          \end{array}
      \right. .
\end{equation*}
Since $S_{y_j}(\by) =  (1-y_j^{-2})S_{1}^{(j)}(\by_{\hat{j}})$ we see that $S_{y_j}(\bw)=0$.  Similarly, one can calculate that $S_{y_jy_j}(\bw) = 2 S_1^{(j)}(\bw)$ and $S_{y_jy_k}(\bw) = 0$ for $j\neq k$, so that the Hessian of $\tilde{f}^{(\bw)}(\mbox{\boldmath$\theta$})$ at $\mathbf{0}$ is a diagonal matrix and
\begin{equation} 
\mathcal{H}_\bw = \frac{2^d}{S(\bw)^d} S_1^{(1)}(\bw) \cdots S_1^{(d)}(\bw). \label{eq:Hessian}
\end{equation}
The proof of Proposition~\ref{prop:minpt} implies that $S_1^{(k)}(\bw) \neq 0$ for any $1\leq k\leq d$, so each $(\bw,t_\bw)$ is non-degenerate.
\end{description}

Thus, we can apply Corollary 3.7 of~\cite{PeWi02} and sum the
expansions~\eqref{eq:Genasm} at each point in $\mV \cap D(\bp)$ to
obtain the asymptotic expansion
\begin{equation} s_n \sim |\mS|^n \sum_{\bw \in \mV \cap D(\bp)} \left(\sum_{l \geq l_\bw} C^{\bw}_l n^{-(d+l)/2}\right) \label{eq:genasm}\end{equation}
for constants $C^{\bw}_l$, where $l_\bw$ is the degree to which $G(\by,t)$ vanishes near $(\bw,t_\bw)$. Since the numerator $G(\by,t) = (1+y_1)\cdots(1+y_d)$ vanishes at all points of $\bw \in \mV \cap D(\bp)$ except for \mbox{$\bp=(\bo,1/|\mS|)$}, the dominant term of~\eqref{eq:genasm} is determined only by the contribution of $\bw=\bp$.  Substituting the value for $\mathcal{H}_{\bp}$ given by Equation~\eqref{eq:Hessian} into Equation~\eqref{eq:domterm} gives the desired asymptotic result.
\end{proof}

\section{Examples}
\label{sec:examples}
We now give two examples, both of which calculate critical points by directly solving the critical point equations.  The first example has only a finite number of critical points, all of which are minimal points.  In contrast, the second example contains a \emph{curve} of critical points (however, as guaranteed by Proposition~\ref{prop:minpt}, no points on this curve are minimal points).

\begin{example}
Consider the model in three dimensions restricted to the positive octant taking the eight steps 
\[ \mS = \{(-1,0,\pm1), (1,0,\pm1), (0,1,\pm1), (0,-1,\pm1)\}. \]
The kernel equation here is
\begin{align*} 
xyz(1-tS(x,y,z))F(x,y,z,t) &= xyz - t y(z^2+1) F(0,y,z) - t x(z^2+1) F(x,0,z)\\
&- t(x^2y+y^2x+y+x)F(x,y,0) \\
& + txF(x,0,0) + tyF(0,y,0),
\end{align*}
with characteristic polynomial
\[ S(x,y,z) = (x+y+\ox+\oy)(z+\oz).\]
The generalized orbit sum method implies $F(1,1,1,t) = \Delta B(x,y,z,t)$ where
\begingroup
\addtolength{\jot}{1em}
\begin{align*}
 B(x,y,z,t) &= \frac{(\ox-x)(\oy-y)(\oz-z)}{\ox\hspace{0.02in}\oy\hspace{0.02in}\oz(1-txyzP(\ox,\oy,\oz))}\cdot \frac{1}{(1-x)(1-y)(1-z)}\\
 &= \frac{(1+x)(1+y)(1+z)}{1-t(z^2+1)(x+y)(xy+1)}.
 \end{align*}
 \endgroup
 Next, we verify that the denominator $H(x,z,y,t)$ of $B(x,y,z,t)$ is
 smooth -- i.e., that $H$ and its partial derivatives don't vanish
 together at any point.  This can be checked automatically by
 computing a Gr\"obner Basis of the ideal generated by $H$ and its
 partial derivatives.

In pseudo-code:\footnote{The input is formatted for Maple version 18.}
 \begin{align*}
&> \quad H := 1-t(z^2+1)(x+y)(xy+1): \\
&> \quad \text{\tt GroebnerBasis}([H, H_x, H_y, H_z, H_t], \text{\tt plex}(t,x,y,z));
\end{align*}
\[ [1] \]
The critical points can be computed:
\[ > \quad \text{\tt GroebnerBasis}([H, tH_t - xH_x, tH_t - yH_y,
tH_t - zH_z], \text{\tt plex}(t,x,y,z));\]
%
\[ [z^2-1, y^2-1, x-y, 8t-y] \]
This implies that there is a finitely minimal critical point $\bp = (1,1,1,1/8)$, where
\[ T(\bp) \cap \mV = \{ (1,1,1,1/8), (1,1,-1,1/8), (-1,-1,1,-1/8), (-1,-1,-1,-1/8) \}. \]
The value of $\mathcal{H}_\bw$ can be calculated at each point to be 1/4.  For instance:\\
\begin{minipage}{\textwidth}
\begin{align*}
&> \quad f := \log S(\bo) - \log S\left(e^{i\theta_1},e^{i\theta_2},e^{i\theta_3}\right): \\
&> \quad \text{\tt subs}\left(\theta_1=0,\theta_2=0,\theta_3=0,
  \text{\tt det(Hessian}\left(f,[\theta_1,\theta_2,\theta_3]\right)\right)); 
\end{align*}
\[ 1/4 \]
\end{minipage}
Equation~\eqref{eq:domterm} then gives the asymptotic result
\[ c_n \sim 4\sqrt{2} \cdot \pi^{-3/2} \cdot n^{-3/2} \cdot 8^n. \]
\end{example}

\begin{example}
Consider the model in three dimensions restricted to the positive octant taking the twelve steps 
\[ \mS = \{(-1,0,\pm1), (1,0,\pm1), (0,1,\pm1), (0,-1,\pm1), (\pm1,1,0),(1,\pm1,0)\}. \]
Now, by our previous analysis, $F(1,1,1,t) = \Delta B(x,y,z,t)$ where
\begin{equation}
 B(x,y,z,t) =
\frac{(1+x)(1+y)(1+z)}{1-t(z^2+1)(x+y)(xy+1)-tz(y^2+1)(x^2+1)}.
\end{equation}

The denominator $H(x,z,y,t)$ of $B(x,y,z,t)$ can again be verified to
be smooth, but the ideal encoding the critical point equations is no
longer zero dimensional; i.e., there are an infinite number of
solutions of the critical point equations.  For instance, the
following calculation shows that any point $(1,-1,z,1/4z)$ with $z \neq 0$ is a non-isolated
critical point:
\begin{align*}
&> \quad H := 1-t(z^2+1)(x+y)(xy+1)-tz(y^2+1)(x^2+1): \\
&> \quad I := \text{{\tt subs}}\left(x=1,y=-1,[H, tH_t - xH_x, tH_t - yH_y, tH_t - zH_z]\right): \\
&> \quad \text{{\tt GroebnerBasis}}(I, \text{{\tt plex}}(t,x,y,z)); 
\end{align*}
\[ [1-4tz] \]
Note that none of these points are minimal -- so
Proposition~\ref{prop:minpt} is not contradicted -- since 
\[\left|(1)\cdot(-1)\cdot(z)\cdot (1/4z)\right| = 1/4 > \frac{1}{|\mS|}.\]
\end{example}

\section{Lower order terms}
\label{sec:lot}
Building upon the work of Pemantle and Wilson, Raichev and
Wilson~\cite{RaWi08} refined the asymptotics of
Equation~\eqref{eq:Genasm} and found expressions for the lower order
constants $C_1,C_2,\dots$, theoretically allowing one to
calculate the contribution of each minimal point $\bw \in \mV \cap
D(\bp)$.  To be explicit, Theorem 3.8 of~\cite{RaWi08} gives the
asymptotic contribution of the minimal point $\bw$ as
\begin{align}
 c_n^{(\bw)} = |\mS|^n \cdot \left[ 2^{-d} \pi^{-d/2} S(\bw)^{d/2} \cdot \left( S^{(1)}_1(\bw) \cdots S^{(d)}_1(\bw) \right)^{-1/2} \right] \cdot n^{-d/2} \cdot &\sum_{k=0}^{N-1} n^{-k} L_k(\tilde{u}^{(\bw)},\tilde{f}^{(\bw)}) \label{eqn:lower} \\
&+ O\left(  |\mS|^n \cdot n^{-(d-1)/2-N} \right), \notag
\end{align}
where, for $\star$ denoting the Hadamard product
 \[(a_1,\dots,a_d) \star (b_1,\dots,b_d) = (a_1b_1,\dots,a_db_d), \]
 we have
\begin{align*}
 \tilde{u}^{(\bw)} (\bt) &:= -\frac{1}{t_\bw} \cdot \frac{G(\bw \star e^{i \bt}, t_\bw)}{H_t(\bw \star e^{i \bt}, t_\bw)} \\
 g_{\bw}(\bt)  &:= \log S(\bw) - \log S(\bw \star e^{i\bt}) - \frac{1}{2}\bt \cdot \tilde{f''}^{(\bw)}(\bt) \cdot \bt^T \\
 L_k(\tilde{u}^{(\bw)},\tilde{f}^{(\bw)}) &:= \sum_{r=0}^{2k} \frac{ \mathcal{D}^{r+k}\left(\tilde{u}^{(\bw)}  \cdot g^r_{\bw} \right)(\bzer)}{(-1)^k 2^{r+k} r! (r+k)!},
 \end{align*}
 and $\mathcal{D}$ is the differential operator
 \[ \mathcal{D} = - \sum_{0 \leq r,s \leq d} \left(\text{Inv}\tilde{f''}^{(\bw)}\right)_{r,s} \partial_{\theta_r} \partial_{\theta_s} = - \frac{S(\bw)}{2} \sum_{r=0}^d \frac{1}{S_1^{(1)}(\bw)}\partial^2_{\theta_r}. \]

 This expression is quite involved -- making it hard to derive a
 \emph{general} asymptotic theorem with lower order terms -- but completely
 effective for a given step set.  The principle difficulty determining
  enumerative results for explicit models in the smooth case
 is the identification of points which actually contribute to the
 asymptotic growth.  In the case of highly symmetric walks this is
 accomplished through the characterization of minimal points given in
 Proposition~\ref{prop:minpt}.

 \begin{example} \label{ex:lower} Consider the two dimensional model
   with step set $\{N,S,NE,SE,NW,SW\} = $ \diag{N,S,NE,SE,NW,SW},
   previously computed to have dominant asymptotics
\[ c_n \sim \frac{\sqrt{6}}\pi \cdot \frac{6^n}n. \]
By Proposition~\ref{prop:minpt}, to find the minimal points we simply need to solve the equation
\[ H(x,y,t) = 1 - t(1+y^2+x+xy^2+x^2+x^2y^2) = 0, \]
in $t$ for all $(x,y) \in \{\pm1\}^2$, and check whether the corresponding solution $t_{x,y}$ satisfies \mbox{$|t_{x,y}| = 1/|\mS| = 1/6$}.  Of the four possible points, we get only two minimal points: the expected point $\bp = (1,1,1/6)$ along with the point $\bs = (1,-1,1/6)$.

Computing the terms in expansion~\eqref{eqn:lower} at these two minimal points -- aided by the Sage implementation of \cite{Raic12} -- gives the asymptotic contributions:
\begin{align*}
c^{(\bp)}_n &=  6^n  \left(\frac{\sqrt{6}}{\pi n} - \frac{17 \sqrt{6}}{16\pi n^2} + \frac{605 \sqrt{6}}{512 \pi n^3} + O(1/n^4)\right) \\
c^{(\bs)}_n &=  (-6)^n \left(\frac{\sqrt{6}}{4\pi n^2} - \frac{33 \sqrt{6}}{64 \pi n^3} + O(1/n^4)\right).
\end{align*}
Thus, the counting sequence for the number of walks of length $n$ has the asymptotic expansion
\[ c_n = 6^n \left( \frac{\sqrt{6}}{\pi n}  -  \frac{\sqrt{6}(17 - 4(-1)^n)}{16\pi n^2} + \frac{\sqrt{6}(38720 - 16896(-1)^n)}{32768\pi n^3} + O(1/n^4) \right). \]
\end{example}

\section{From diagonals to differential equations}
\label{sec:tele}

As seen in Corollary~\ref{cor:Dfinite}, the generating function
$F(\bo,t)$ will be D-finite for any highly symmetric model $\mS$.
Indeed, from the expression $F(\bo,t) = \Delta G(\bz,t)/H(\bz,t)$ it
is possible in principle to compute an annihilating linear
differential equation of $F(\bo,t)$ through the use of algorithms for
creative telescoping.  These algorithms, which are typically grouped
into those that perform elimination in an Ore algebra -- including the
famous algorithm of Zeilberger~\cite{ Zeil90} -- and those which use
an ansatz of undetermined coefficients, compute differential operators
annihilating multivariate integrals and connect to diagonals of
rational functions through the relations
\begin{align}
  \frac{1}{2\pi i} \int_\Omega \frac{B(z_1,z_2/z_1,z_3,\dots,z_d,t)}{z_2} d z_2 &= \Delta_{1,2} B(\bz,t) \\
  \left(\frac{1}{2\pi i}\right)^d \int_T \frac{B(z_1,z_2/z_1,z_3/z_2,\dots,z_d/z_{d-1},t/z_d)}{z_1 z_2 \cdots z_d} d\bz  &= \Delta B(\bz,t),
\end{align}
where $B(\bz,t)$ is analytic in a neighbourhood of the origin,
$\Omega$ is an appropriate contour in $\mathbb{C}$ containing the
origin, and $T$ is an appropriate torus in $\mathbb{C}^d$ containing
the origin.  The reader is directed to~\cite{Kout10} and
\cite{BoLaSa13} for details on how these methods work and are
implemented in modern computer algebra systems.  In
Table~\ref{tab:ode} we have computed annihilators for the four highly
symmetric models in two dimensions using an ansatz method developed
and implemented in Mathematica by Koutschan~\cite{Kout10}. 

\begin{table}
\centering
\begin{tabular}{ | c | r | }
  \hline
   $S$ & Annihilating DE \\ \hline
  &\\[-5pt] 
  \diag{N,S,E,W}  & $t^2(4t-1)(4t+1)D_t^3+2t(4t+1)(16t-3)D_t^2$\\
  &$+\left(-6+28 t+224 t^2\right)D_t+(12+64 t)$ \\ [+3mm]  
  \diag{NE,SE,NW,SW} & $t^2(4t+1)(4t-1)^2 D_t^3+t(4t-1)(112t^2-5)D_t^2$ \\
  & $+4(8t-1)(20t^2-3t-1)D_t+\left(-4-48 t+128 t^2\right)$  \\ [+3mm]
  \diag{N,S,NE,SE,NW,SW} & $t^2(6t-1)(6t+1)(2t+1)(2t-1)(12t^2-1)D_t^3$\\
  &$+t(2t-1)(6048t^5+2736t^4-672t^3-336t^2+6t+5)D_t^2$\\
  &$+\left(-4+16 t+516 t^2+96t^3-5520 t^4-2304t^5+17280 t^6\right)D_t$\\
  &$+\left(8+132 t+96 t^2-1104 t^3-1152 t^4+3456 t^5\right)$ \\[+3mm]
  \diag{N,S,E,W,NW,SW,SE,NE} & $-t^2(4t+1)(8t-1)(2t-1)(t+1)D_t^3$\\
  &$+t\left(-5+33 t+252 t^2-200 t^3-576 t^4\right)D_t^2$ \\
  &$+\left(-4+48 t+468 t^2-88 t^3-1152t^4\right)D_t$\\
  &$+\left(12+144 t+72 t^2-384 t^3\right)$ \\
\hline
\end{tabular}\\[2mm]

\caption{Annihilating differential equations for the highly symmetric quarter plane models.} \label{tab:ode}
\end{table}

\begin{table}
\centering
\begin{tabular}{ | c | r | }
  \hline
   Dimension $d$ & Annihilating DE \\ \hline
  &\\[-5pt] 
  3  & $-t^3 (2 t-1) (2 t+1) (6 t-1) (6 t+1) D_t^4$\\
&$-4 t^2 \left(576 t^4+36 t^3-140 t^2-5 t+3\right) D_t^3$ \\
&$-4 t \left(2592 t^4+324 t^3-531 t^2-40 t+9\right) D_t^2$ \\
&$-8\left(1728 t^4+324 t^3-282 t^2-34 t+3\right) D_t$ \\
&$-24 \left(144 t^3+36 t^2-17 t-3\right)$ \\ [+3mm]  
  4 & $-t^4 (4 t-1) (4 t+1) (8 t-1) (8 t+1) D_t^5$\\
&$-4 t^3 (4 t+1) \left(1536 t^3-320 t^2-30 t+5\right) D_t^4$\\
&$-4 t^2 \left(47104 t^4+3968 t^3-2976 t^2-145 t+30\right)D_t^3$\\
&$-12 t \left(45056 t^4+5760 t^3-2368 t^2-191 t+20\right) D_t^2$\\
&$-24 \left(21504 t^4+3712 t^3-848 t^2-106 t+5\right) D_t$\\
&$-96 \left(1024 t^3+224 t^2-24t-5\right)$\\ 
\hline
\end{tabular}\\[2mm]

\caption{Annihilating differential equations for the models $\{e_1,-e_1,\dots,e_d,-e_d\}$.} \label{tab:ode2}
\end{table} 

Given an annihilating linear differential operator of the univariate generating function $F(\bo,t)$, one can easily compute a linear recurrence relation that the counting sequence $(c_n)$ must satisfy. The Birkhoff-Trjitzinsky method (see \cite{ WiZe85} and \cite{ FlSe09}) can then be used to determine a basis of solutions to this recurrence.  Each element of the basis has dominant asymptotic growth of the form
\[ c_n^{(k)} \sim C_k \rho^n n^{\beta_k} (\log n)^{l_k}, \] 
for computable constants $C_k, \rho, \beta_k, l_k$.  Using this technique
to approach an asymptotic analysis for lattice walks in restricted
regions has been used previously -- for instance in the work of Bostan
and Kauers~\cite{ BoKa09} on two dimensional lattice walks confined to
the positive quadrant -- however it is not apparent how the number of
walks in a model, $c_n$, is represented as a linear combination of the
basis elements $c_n^{(k)}$.  Determining this linear combination is
known in the literature as the \emph{connection problem}, as it
describes how the generating function is connected to a local basis of
singular solutions.  This highlights a severe drawback to using the
differential equation for asymptotics, when compared to the methods of
this section: there is no known effective procedure to solve the connection problem
 in general, even when the coefficients of the differential equation are
known to be rational functions (the connection problem is believed by some to be uncomputable~\cite{ FlSe09}).
In essence, this implies that the
multiplicative growth constant of the dominant asymptotic term cannot
be determined rigorously in general (Bostan and Kauers used numerical
approximations to non-rigorously solve the connection problem for
their work on two dimensional models).

\begin{example}
  As seen in Table~\ref{tab:ode}, the step set univariate generating
  function $\sum c_n t^n$ of the quarter-plane model
  $\{N,S,NE,SE,NW,SW\} = $~\diag{N,S,NE,SE,NW,SW} is annihilated by
  the differential operator
\begin{align*} 
\mathcal{L} &=  \left(-t^2+52 t^4-624 t^6+1728 t^8\right)D_t^3+\left(-5 t+4 t^2+348 t^3-4080 t^5-576 t^6+12096 t^7\right)D_t^2\\
  &+\left(-4+16 t+516 t^2+96t^3-5520 t^4-2304t^5+17280 t^6\right)D_t\\
  &+\left(8+132 t+96 t^2-1104 t^3-1152 t^4+3456 t^5\right),
\end{align*}
which implies that the sequence $(c_n)$ satisfies the following linear recurrence relation with polynomial coefficients
\begin{align*}
0 &= \left(-n^3 - 20n^2 - 133n - 294\right) c_{n+6} + \left(4n^2+ 52n + 168\right) c_{n+5}+ \left(52n^3 + 816n^2 + 4304n +7620\right) c_{n+4} \\
&+ \left(96n + 384\right) c_{n+3}
+ \left(-624n^3 - 5952n^2 - 19008n - 20304\right) c_{n+2} 
+ \left(-576n^2 - 2880n - 3456\right) c_{n+1} \\
&+ \left(1728n^3 + 6912n^2+ 8640n + 3456\right)c_n.
\end{align*}

Using the Birkhoff-Trjitzinsky method one computes a basis of local solutions at infinity to this degree six linear recurrence relation (the basis given here was computed using the Sage package of \cite{KaJaJo14}):

\begin{align*} 
c_n^{(1)} &= \frac{6^n}{n} \left(1 -\frac{17}{16}n^{-1} + \frac{605}{512}n^{-2} + O\left(n^{-3}\right)\right) 
&&
c_n^{(2)} = \frac{6^n}{n^2} \left(1 -\frac{33}{16}n^{-1} + \frac{1565}{512}n^{-2} + O\left(n^{-3}\right)\right)  
\\
c_n^{(3)} &= \frac{(2\sqrt{3})^n}{n^4}\left(1 - \frac{14+3\sqrt{3}}{2}n^{-1} + O\left(n^{-2}\right)\right) 
&&
c_n^{(4)} = \frac{(-2\sqrt{3})^n}{n^4}\left(1 - \frac{14-3\sqrt{3}}{2}n^{-1} + O\left(n^{-2}\right)\right) 
\\
 c_n^{(5)} &= \frac{2^n}{n^3}\left(1 - \frac{51}{16}n^{-1} + \frac{3341}{512}n^{-2} +O\left(n^{-3}\right)\right) 
&&
c_n^{(6)} = \frac{(-2)^n}{n^2}\left(1 - \frac{35}{16}n^{-1} + \frac{1805}{512}n^{-2} +O\left(n^{-3}\right)\right),
\end{align*}
so that $c_n = O(6^n/n)$.  Note that the results of Example~\ref{ex:lower} imply
\[ c_n = \frac{\sqrt{6}}{\pi}c_n^{(1)} + \frac{\sqrt{6}}{4\pi}c_n^{(2)} + O\left((2\sqrt{3})^n\right),   \]
and we can partially resolve the connection problem, however this is only possible because leading term asymptotics for $c_n$ were already calculated through the techniques of Pemantle, Raichev, and Wilson.
\end{example}

Although differential operators are very useful data structures for the
D-finite functions which they annihilate, the work above illustrates
that the representation of~$F(\bo,t)$ as a rational diagonal can yield
easier access to its asymptotic information when coupled with the
results of analytic combinatorics in several variables (at least in
the smooth case).  Furthermore, the combinatorial properties of
lattice path models often naturally give representations of their
generating functions as rational diagonals, and determining
annihilating differential operators for these diagonals can be
difficult.  Creative telescoping methods -- although always improving
(see, for example,~\cite{ BoLaSa13}) -- do not scale well with degree
and must be calculated on a model by model basis.

\section{Walks in a Weyl Chamber}
\label{sec:Weyl}

In 1992, Gessel and Zeilberger~\cite{GeZe92} outlined an extension of
the reflection principle -- originally used by Andr{\'e}~\cite{ An87}
in the nineteenth century to solve the two candidate ballot problem --
to lattice walks on regions preserved under the actions of
Coxeter-Weyl finite reflection groups.  In this section we show how
the highly symmetric walks can be viewed in this context.  In addition
to giving an alternative view of the calculations
presented through the kernel method in Section~\ref{sec:lwalks}, this
view also allows us to determine diagonal expressions for the
excursion generating function and permits a segu\"e to  a discussion of how other, non-highly symmetric
models, fit into this template.

\subsection{Weyl Chambers and Reflectable Walks}
The following definitions are taken from Gessel and Zeilberger~\cite{GeZe92}, Grabiner and Magyar~\cite{GrMa93}, and Humphreys~\cite{Hum72}, and the reader is directed to these manuscripts for more details.  

A \emph{(reduced) root system} is a finite set of vectors $\Phi \subset \mathbb{R}^n$ such that
\begin{itemize}
\item for any $x,y \in \Phi$, the set $\Phi$ contains the reflection of $y$ through the hyperplane with normal~$x$ 
\[\sigma_x(y) = y - 2\frac{(x,y)}{(x,x)}x; \]
\item for any $x,y \in \Phi$, $x-\sigma_y(x)$ is an integer multiple of $y$;
\item the only scalar multiples of $x \in \Phi$ to be in $\Phi$ are $x$ and $-x$.
\end{itemize}

The set of linear transformations generated by the reflections
$\sigma_x$ is always a finite Coxeter group and is called the
\emph{Weyl group} $W$ of the root system.  The complement of the union
of the hyperplanes whose normals are the root system is an open set,
and a connected component of this open set is called a \emph{Weyl
  chamber}.  For the root system $\Phi$, a set of \emph{positive
  roots} $\Phi^+$ is a subset of $\Phi$ such that
\begin{enumerate}
\item for each $x \in \Phi$ exactly one of $x$ and $-x$ is in $\Phi^+$;
\item for any two distinct $\alpha,\beta \in \Phi$ such that $\alpha + \beta$ is a root, $\alpha + \beta \in \Phi^+$. 
\end{enumerate}
An element of $\Phi^+$ is called a \emph{simple root} if it cannot be
written as a sum of two elements of $\Phi^+$, and a maximal set
$\Delta$ of simple roots is called a \emph{basis} for the root system.
It can be shown that for a basis $\Delta$ any $x \in \Phi$ is a linear combination of
members of $\Delta$ with all non-negative or non-positive
coefficients, and that the set $\{\sigma_x:x \in \Delta\}$ generates
the Weyl group $W$.

Fix a root system $\Phi$ and a basis $\Delta$, and let 
\begin{itemize}
\item $\mS \subset \mathbb{Z}^n$ be a set of steps such that $W\cdot
  \mS = \mS$ -- i.e., $\mS$ is preserved under each element of the
  Weyl group;
\item $L$ be a lattice, restricted to the linear span of elements of
  $\mS$, such that $W\cdot L = L$;
\item $C$ be the Weyl chamber
\[ C = \{ \bz \in \mathbb{R}^n : (\alpha,\bz) > 0 \text{ for all } \alpha \in \Delta\}. \]
\end{itemize}
The lattice path model in the Weyl chamber $C$ using the steps $\mS$ beginning at a point $\mathbf{a} \in C$ is the combinatorial class of all sequences of steps in $\mS$ beginning at $\mathbf{a}$ and never leaving $C$ (when viewed as a walk on $L$ in the typical manner).  If, in addition to the requirements above, the two conditions
\begin{enumerate}
\item For all $\alpha \in \Delta$ and $s \in \mS$, $(\alpha,s) =\pm k(\alpha)$ or 0, where $k(\alpha)$ is a constant depending only on~$\alpha$;
\item For all $\alpha \in \Delta$ and $\lambda \in L$, $(\alpha,\lambda)$ is an integer multiple of $k(\alpha)$ depending only on $\alpha$;
\end{enumerate}
are met, we say that the lattice path model is \emph{reflectable}, and any step $s \in S$ taken from any lattice point inside $C$ will not leave $C$ except possibly to land on its boundary (one of the hyperplanes whose normals are the elements of $\Phi$).

The main result of Gessel and Zeilberger~\cite{GeZe92}, after a conversion from constant term extraction to diagonal extraction, is the following.

\begin{theorem}[Gessel and Zeilberger~\cite{GeZe92}] \label{thm:GeZe}
 Given a reflectable walk as defined above such that $(a,\alpha)$ is an integer multiple of $k(\alpha)$ for each $\alpha \in \Delta$, and an element $b \in C$ such that $(b,\alpha)$ is also an integer multiple of $k(\alpha)$ for each $\alpha \in \Delta$, the generating function for the number of walks which begin at $a$, end at $b$, and stay in $C$ is
\begin{equation} F_{a\rightarrow b}(t) = \Delta \left[ \frac{1}{1-t(z_1\cdots z_d)S(\bz)} \cdot \bz^{-\mathbf{b}} \cdot \sum_{w \in W} (-1)^{l(w)} \bz^{w(\mathbf{a})}\right], \label{eq:RefB} \end{equation}
where $l(w)$ is the minimal length of $w$ represented as a product of elements in $\{\sigma_x:x \in \Delta\}$.
\end{theorem}

If $(b,\alpha)$ is an integer multiple of $k(\alpha)$ for each $\alpha \in \Delta$ and $b \in C$, and the formal power series $\sum_{\mathbf{b} \in C} \bz^{-\mathbf{b}}$ exists (see~\cite{ApKa13} for a discussion on the existence of multivariate Laurent series) then summing Equation~\eqref{eq:RefB} over all possible endpoints implies that the generating function for the number of walks beginning at $a$ and staying in $C$ which are allowed to end anywhere is
\begin{equation} F_a(t) = \Delta \left[ \frac{1}{1-t(z_1\cdots z_d)S(\bz)} \cdot \sum_{\mathbf{b} \in C} \bz^{-\mathbf{b}} \cdot \sum_{w \in W} (-1)^{l(w)} \bz^{w(\mathbf{a})}\right]. \label{eq:RefAny} \end{equation}

\subsection{Classification of Weyl chambers and reflectable walks}
Given two root systems \mbox{$\Phi_1 \subset \mathbb{R}^n$} and $\Phi_2 \subset \mathbb{R}^m$, one can create a new root system $\Phi_1\times\Phi_2$ by treating the two vector spaces spanned by the elements of $\Phi_1$ and $\Phi_2$ as mutually orthogonal subspaces of $\mathbb{R}^{n+m}$.  To this end, a root system $\Phi$ is called \emph{reducible} if it can be decomposed as $\Phi = \Phi_1 \cup \Phi_2$, where $\Phi_1$ and $\Phi_2$ are root systems whose elements are pairwise orthogonal, and \emph{irreducible} otherwise.  

One of the main results in the study of root systems -- which arises in relation to Lie algebras and representation theory -- is a complete classification of the irreducible root systems, consisting of four infinite families ($A_n$ for $n \geq 1$, $B_n$ for $n \geq 2$, $C_n$ for $n \geq 3$, and $D_n$ for $n\geq 4$) and five exceptional cases ($E_6,E_7,E_8,F_4,$ and $G_2$).  The interested reader is directed to Section 11.4 of Humphreys~\cite{Hum72} for details and a proof of the classification. 

\begin{example} \label{ex:A1d}
There is, up to scaling by a constant, one root system in $\mathbb{R}$: the system $\Phi_1 = \{\pm1\}$ with basis $\Delta_1 = \{1\}$, which is called $A_1$.  From this, the root system $A_1\times A_1 = A_1^2 \subset \mathbb{R}^2$ is defined as the direct sum of two copies of $A_1$, giving elements $\Phi_2 = \{\pm e_1,\pm e_2\}$ and basis $\Delta_2 = \{e_1,e_2\}$.  In general, for any $d \in \mathbb{N}$ the root system $A_1^d$ will be the system with elements $\Phi = \{\pm e_1, \dots, \pm e_d\}$, which admits the basis $\Delta = \{e_1,\dots,e_d\}$.
\end{example}

\subsection{Highly symmetric walks are walks in Weyl chambers}
The root system $A_1^d$, described in Example~\ref{ex:A1d}, has corresponding Weyl chamber
\[ C = \{ \bz : z_1>0 \text{ and } z_2>0 \text{ and } \cdots \text{ and } z_d>0 \} = \left(\mathbb{Z}_{>0}\right)^d, \] 
and it follows directly from the definitions above that a step set $\mS \subset \mathbb{Z}^d$ is a reflectable walk with respect to $\Delta$ if and only if it is highly symmetric.  As $C$ does not include the the hyper-planes $\{z_1=0\}, \cdots, \{z_d=0\}$, we shift the origin of the walks under consideration by starting them at the point $\mathbf{a} = \bo$.  The Weyl group $W$ corresponding to this set of roots is isomorphic to $\mathbb{Z}_2^d$ (in fact, it is equal to the group $\mathcal{G}$ as defined in Section~\ref{sec:orbit}) and
\begin{align*}
\sum_{\mathbf{b} \in C} \bz^{-\mathbf{b}} &= \frac{1}{z_1-1} \cdots \frac{1}{z_d-1} \\
\sum_{w \in W} (-1)^{l(w)} \bz^{w(\mathbf{a})} &= (z_1-\oz_1)\cdots (z_d-\oz_d).
\end{align*}
Substitution into Equation~\ref{eq:RefAny} recovers Equation~\eqref{eq:rat}, shifted by a factor of $(z_1\cdots z_d)$ to account for the shifted walk origin $\mathbf{a}$:
\[ F_{\bo}(t) = \Delta \left[ \frac{(1+z_1)\cdots(1+z_d)}{1-t(z_1\cdots z_d)S(\bz)} \cdot (z_1\cdots z_d)\right]. \]
We note that the argument presented in Section~\ref{sec:lwalks} -- which is a standard generalization of the kernel method -- mirrors the proof of Theorem~\ref{thm:GeZe} given by Gessel and Zeilberger.

\subsection{Excursions}
\label{sec:excursions}
Not only can we recover previous results, but we can now give asymptotics for the number of walks which return to the origin.  Taking $\mathbf{a}=\mathbf{b}=\mathbf{1}$ in Equation~\ref{eq:RefB}, we see that the number of excursions $e_n$ is given by
\begin{align*}
e_n &= [t^n] \Delta \left( \frac{(z_1-\oz_1) \cdots (z_d-\oz_d)}{1-t(z_1\cdots z_d)S(\bz)} \cdot (z_1\cdots z_d)^{-1}\right) \\
&= [t^n] \Delta \left( \frac{t^2(z_1^2-1) \cdots (z_d^2-1)}{1-t(z_1\cdots z_d)S(\bz)} \cdot (tz_1\cdots z_d)^{-2}\right) \\
&= [t^{n+2}] \Delta \left( \frac{t^2(z_1^2-1) \cdots (z_d^2-1)}{1-t(z_1\cdots z_d)S(\bz)}\right).
\end{align*}

Note that the form of the final rational function on the right hand
side implies that the same minimal points will appear in the analysis
of excursion asympotics -- however, due to the factors of
$(z_1-1)\cdots(z_d-1)$ now present in the numerator the finitely
minimal point $\rho = (1,...,1,1/|S|)$ will vanish, bringing down the
polynomial growth factor of excursions compared to the asymptotics of
walks ending anywhere.  Furthermore, as more than one minimal point
can now determine the dominant asymptotics closed form results are not
easily obtainable.  Despite that, as the minimal points are still
classified by Proposition~\ref{prop:minpt}, one can use the machinary
available to calculate lower terms in asymptotic expansions (as in
Section~\ref{sec:lot}) to determine the asymptotics of specific
models.

\begin{example}
Consider the highly symmetric 2D step set $\{N,S,NE,SE,NW,SW\} = $ \diag{N,S,NE,SE,NW,SW}.  Here we have
\[ e_n = [t^{n+2}] \Delta \left( \frac{t^2(x^2-1)(y^2-1)}{1-(tx^2y^2+ty^2+tx^2+t+txy^2+tx)}  \right), \]
and as discussed in Example~\ref{ex:lower} this rational function has the expected minimal point $\bp = (1,1,1/6)$ along with the point $\bs = (1,-1,1/6)$.  Computing the terms in expansion~\eqref{eqn:lower} at these two minimal points -- again aided by the Sage implementation of \cite{Raic12} -- gives the asymptotic contributions (after properly shifting index):
\[ e^{(\bp)}_n =  6^n  \left(\frac{3\sqrt{6}}{2\pi n^3} + O(1/n^4)\right) \qquad\qquad 
e^{(\bs)}_n =  (-6)^n \left(\frac{3\sqrt{6}}{2\pi n^3} + O(1/n^4)\right). \]
Thus, the counting sequence for the number of excursions of length $n$ has the asymptotic expansion
\[ e_n = 6^n \left( \frac{3\sqrt{6}}{2\pi n^3}(1 + (-1)^n) + O(1/n^4) \right), \]
where we note that there are no excursions of odd length.
\end{example}

As the denominator of the rational function under consideration is smooth, and the numerator $t^2(z_1^2-1) \cdots (z_d^2-1)$ vanishes at any minimal point to order $d$, the asymptotic expansion given in Equation~\eqref{eqn:lower} implies the following.

\begin{theorem}
\label{thm:excursion}
Let $\mS \subseteq \{-1,0,1\}^d \setminus \{\mathbf{0}\}$ be a set of unit steps in dimension $d$.  If $\mS$ is symmetric with respect to each axis, and $\mS$ takes a positive step in each direction, then the number of walks $e_n$ of length $n$ taking steps in $\mS$, beginning and ending at the origin, and never leaving the positive orthant satisfies
\[ e_n = O\left(\frac{|\mS|^n}{n^{3d/2}}\right). \]
\end{theorem}

\section{Conclusion}
\label{sec:conc}

The purpose of this article, aside from the specific combinatorial
results it contains, is to reinforce the notion that there are many
possibilities for studying lattice walks in restricted regions through
the use of diagonals and analytic combinatorics in several
variables: in this context the diagonal data structure often permits
analysis in general dimension. Furthermore, walks with symmetry across each axis all have a smooth
singular variety, making them the perfect entry point to this
confluence of the kernel method, the reflection principle and analytic combinatorics of several variables.

\subsection{Generalizations: Other Weyl Chambers}
A major goal moving forward is to deal with more general step set models.
As a first attempt, we have also considered reflectable walks in~$A_2$, and
some other related models, and this gives nice diagonal expressions
for the generating functions for the models with group order 6 in the
classification of~\cite{BoMi10}. However, the expressions are far
more difficult to analyze with these asymptotic techniques, since the
expressions no longer fall in the simplest, smooth case. 

More generally, Grabiner and Magyar~\cite{GrMa93} have classified, for
each irreducible root system $\Phi$, the step sets which give rise to
a reflectible lattice path model in the corresponding Weyl chamber.
This combinatorial classification gives a large collection of future
objects to study through the means of analytic combinatorics in
several variables. Assuming one can get the generating function for the number
of walks in a more general setting as a rational diagonal, results on
asymptotics can be reduced to an analysis of this rational function.
Both~\cite{PeWi13} and~\cite{RaWi11} give results for singular
varieties which are non-smooth, but whose critical points are
\emph{multiple points}.  Due to the constraints on the rational
functions arising from the combinatorial nature of lattice paths in
restricted regions, there is hope for a completely systematic treatment which
allows for some non symmetries.  

This leads to the natural question, can the infamous Gessel walks be
expressed as walk in a Weyl Chamber? A positive answer could result in a far
simpler path to a generating function expression than those presently
known, even the methods explicitly derived by humans~\cite{BoKuRa14},
and a negative answer might help explain why it has resisted simpler
approaches. 

Furthermore, the asymptotic enumeration of excursions has received much
attention lately, due to the recent work of Denisov and Wachtel. It
could be interesting to link their work to expressions using
diagonals in the case of D-finite models. The results
of~\cite{BoRaSa14} suggest very compelling evidence that the boundary
between D-finite models and non-D-finite models leaves strong traces in
the asymptotic enumeration.

\subsection{Are all D-finite models diagonals?}
Across the study of lattice path models to date, it has been true that
every model with a D-finite generating function is accompanied by an
expression of the generating function as a diagonal of a rational
function (or equivalent).  A conjecture of Christol~\cite{ Chri90}
posits that any globally bounded D-finite function (which includes
power series convergent at the origin with integer coefficients) can
be written as the diagonal of a multivariate rational function. Could
one prove a lattice path version of this conjecture? More practically,
could such a result be made effective with an automatic
method of writing known D-finite functions as diagonals?

Finally, it would be interesting to understand if there is a direct
combinatorial interpretation for the diagonal operator acting of
rational functions. Recent work of Garrabrant and Pak~\cite{GaPa14}
gives a tiling interpretation of diagonals of $\mathbb{N}$-rational
functions. Our rationals here are very combinatorial, although they
have some negative coefficients. Very possibly a signed version of
their construction might capture the diagonals that we build.

\section{Acknowledgments}
The authors would like to thank Manuel Kauers for the construction in
Proposition~\ref{prop:diag}, and illuminating discussions on diagonals
of generating functions, and the anonymous referees of an extended abstract
of this work for their comments and suggestions. 
We are also grateful to Mireille Bousquet-M\'elou for
pointing out some key references. 

\bibliographystyle{plain}
\bibliography{bibl}

\end{document}